\documentclass[]{amsart}

\usepackage{hyperref}
\usepackage{amssymb}
\usepackage{amsthm}
\usepackage{bigints}

\newtheorem{theorem}{Theorem}[section]
\newtheorem{corollary}[theorem]{Corollary}
\newtheorem{proposition}[theorem]{Proposition}
\newtheorem{definition}[theorem]{Definition}
\newtheorem{lemma}[theorem]{Lemma}

\makeatletter
\newtheorem*{rep@theorem}{\rep@title}
\newcommand{\newreptheorem}[2]{%
\newenvironment{rep#1}[1]{%
 \def\rep@title{#2 \ref{##1}}%
 \begin{rep@theorem}}%
 {\end{rep@theorem}}}
\makeatother

\newreptheorem{theorem}{Theorem}
\newreptheorem{definition}{Definition}

\theoremstyle{definition}
\newtheorem{example}[theorem]{Example}
\newtheorem{remark}[theorem]{Remark}

\newcommand{\Linf}{L$_\infty$}
\newcommand{\MC}{\operatorname{MC}}
\newcommand{\Hom}{\operatorname{Hom}}
\newcommand{\im}{\operatorname{im}}
\newcommand{\homology}{\operatorname{H}}

\title[Minimal models of quantum \Linf-algebras via the BV-formalism]{Minimal models of quantum homotopy Lie algebras via the BV-formalism}

\author{Christopher Braun}
\address{Department of Mathematics and Statistics\\
Lancaster University\\
Lancaster LA1 4YF\\United Kingdom}
\email{c.braun@lancaster.ac.uk}

\author{James Maunder}
\address{Max-Planck-Institut f\"ur Mathematik\\
Vivatsgasse 7\\
53111 Bonn\\
Germany}
\email{maunder@mpim-bonn.mpg.de}

\begin{document}

\begin{abstract}
Using the BV-formalism of mathematical physics an explicit construction for the minimal model of a quantum \Linf-algebra is given as a formal super integral. The approach taken herein to these formal integrals is axiomatic, and they can be approached using perturbation theory to obtain combinatorial formulae as shown in the appendix. Additionally, there exists a canonical differential graded Lie algebra morphism mapping formal functions on homology to formal functions on the whole space. An inverse \Linf-algebra morphism to this differential graded Lie algebra morphism is constructed as a formal super integral.
\end{abstract}

\maketitle

\section*{Introduction}

The Batalin-Vilkovisky (BV-)formalism was originally introduced in physics as a tool to quantise gauge theories and is named after the creators Igor Batalin and Grigori Vilkovisky \cite{batalin_vilkovisky_gauge_algebra_quantization,batalin_vilkovisky_quantization_of_gauge_algebras}. One of the strengths of the BV-formalism is it describes how to deal with certain super path integrals, understood as formal power series using perturbation theory. The BV-formalism has also found success in other fields, leading to many results including: deformation quantisation \cite{cattaneo_felder_deformation}, an alternative description of the graph complex \cite{hamilton_laz_graph}, an alternative proof of the Kontsevich theorem \cite{intro_to_graded_geo_BV_formalism_and_applications}, and manifold invariants \cite{cattaneo_mnev}. All told, the BV-formalism provides a framework in which odd symplectic geometry, homological algebra, and path integrals interact successfully. The geometric formulation of the BV-formalism was pioneered by Khudaverdian, cf.~\cite{khudaverdian_semidensities,khudaverdian_geometry_of_superspace,khudaverdian_BV_and_odd_symp_geo,khudaverdian_nersessian}. A modern formulation of the BV-formalism was given by Schwarz \cite{schwarz}, but it should be noted that there are many papers where the BV-geometry is considered from various standpoints, see \cite{getzler_BV,khudaverdian_semidensities,khudaverdian_voronov,kosmann-schwazbach_monterde,severa} for example. BV-algebras themselves have also been generalised to $\operatorname{BV}_\infty$-algebras: cf.\ \cite{bashkirov,braun_laz_homotopy_BV,homotopy_BV_algebras,kravchenko_defo_BV_algebras}.

The formal geometry of the BV-formalism is well suited to the study of (quantum) \Linf-algebras which can be studied in the same language. Accordingly, this viewpoint is taken herein to construct \emph{minimal models} of quantum \Linf-algebras, see Definition \ref{def_minimal}.

Quantum \Linf-algebras arose in work of Zwiebach \cite{zwiebach} on closed string field theory. They have appeared in work of Markl \cite{markl_loop_homotopy} under the name `loop homotopy algebras', and have appeared in work of the first author joint with Lazarev \cite{braun_laz_unimodular}. Quantum \Linf-algebras are a `higher genus' version of a cyclic \Linf-algebra: the definition is recalled in Section \ref{sec_linf}. One particularly amenable viewpoint of a (quantum) \Linf-algebra structure on a super vector space is a solution to the Maurer-Cartan equation in an appropriate differential graded Lie algebra. The Maurer-Cartan equation is known in physics---more specifically in quantum field theory---as the (quantum or classical) master equation and its use is central to this paper. Indeed, in this language, a quantum \Linf-algebra is a solution to the quantum master equation just as a cyclic \Linf-algebra is a solution to the classical master equation. Maurer-Cartan elements have many uses in mathematics besides defining \Linf-algebra structures: they govern deformation functors \cite{braun,manetti,maunder_koszul_duality,pridham,schlessinger_stasheff_deformation_rational}, in certain cases correspond to morphisms of certain commutative differential graded algebras, and model rational topological spaces \cite{hinich_stacks,andrey_MC,laz_markl,maunder_unbased_rat_homo,quillen}

It is known that minimal models exist for many sorts of homotopy algebras \cite{kadeishvili,kontsevich}. To prove existence and uniqueness of minimal models is usually fairly straightforward. Indeed, in this paper, this is the content of Proposition \ref{prop_filtered_quasi_iso}, which we obtain as a consequence of standard facts concerning Maurer-Cartan moduli sets of differential graded Lie algebras. However, this argument is not at all constructive---one is often concerned not just with the existence of minimal models but also wishes to have explicit formulae to hand.

Explicit formulae for the structure maps of minimal models for $A_\infty$-algebras were given in \cite{kontsevich_soibelman,markl_transferring,merkulov} as sums over trees. A more general approach was taken in \cite{chuang_laz_feynman} where an explicit formula for minimal models, in terms of sums over `stable graphs', for an algebra over the cobar-construction of a differential graded modular operad was constructed.

We take a different approach to deducing formulae for minimal models in the present paper. Indeed, the formulae in terms of stable graphs are reminiscent of those given by perturbative expansions of path integrals using Feynman diagrams and this is the perspective we pursue here. We will show that the minimal model of a quantum \Linf-algebra can be given by a simple explicit integral formula coming from the BV-formalism (Theorem \ref{thm_main}). These sorts of integrals have already been studied in the context of quantum field theory \cite{cattaneo_mnev,costello}. The advantage of this approach is that we obtain a simpler and more conceptual proof of the minimal model formulae. 

In fact, the combinatorial formulae in terms of stable graphs that one obtains using the results of \cite{chuang_laz_feynman} can be deduced from the integral formula given in this paper by standard methods of expansions of path integrals in terms of Feynman diagrams: this is done in Appendix \ref{sec_appendix}.

Recently, minimal models have found applications in theoretical physics. More specifically, minimal models of A$_\infty$-algebras have applications in (open) string field theory and quiver gauge theory \cite{aspinwall_fidowski,aspinwall_katz,kajiuara,lazaroiu,tomasiello}, whereas (quantum) \Linf-algebras and their minimal models have applications in closed string field theory \cite{meunster_sachs_quantum_open-closed,muenster_sachs_homotopy_classification}, note that quantum \Linf-algebras are called `loop homotopy algebras' loc.~cit. Moreover, the approach taken by the BV-formalism has been extended towards applications in non-commutative homotopical algebra \cite{barannikov_matrix_de_rham,barannikov_noncom_BV,barannikov_solving_nocom_BV,hamilton_classes_on_compactifications,hamilton_noncomm_geo_and_compactifications}. This paper is concerned with the finite-dimensional case (and is, therefore, not physical), but the objects and methods used herein frequently appear in physics literature. Indeed, despite not being strictly physical, Cattaneo and Mnev work with a finite-dimensional `toy model' of Chern-Simons field theory \cite{cattaneo_mnev}. The infinite-dimensional case is a somewhat different aspect and is treated in \cite{costello}, for example. It should also be noted that, under certain conditions, infinite dimension algebras can be reduced to equivalent finite dimensional algebras suitable for many purposes. One example of this is in \cite[Theorem 8.3]{braun_laz_unimodular}, which uses a theorem of Lambrechts and Stanley \cite{lambrechts_stanley} rephrased and applied to the context of Chern-Simons theory.'

The paper is organised as follows. Section \ref{sec_formal_geo} recalls results of linear formal odd symplectic geometry. In particular, the master equation (or Maurer-Cartan equation) and the notion of a strong deformation retract from one odd symplectic vector space to another are both recalled. Section \ref{sec_bv} introduces the theory of integration used in this setting (Definition \ref{def_integral}), provides a proof of an analogue of Stokes Theorem (Proposition \ref{prop_BV_stokes}), and discusses the relevant parts of the BV-formalism. Section \ref{sec_linf} recalls details surrounding the theory of (quantum) \Linf-algebras and Proposition \ref{prop_weight_grading} introduces an important and useful filtration. The main result of the paper (Theorem \ref{thm_main}) is contained within Section \ref{sec_minimal_models}. That is, the explicit construction of the minimal model for a given quantum \Linf-algebra via a formal super integral is given in Section \ref{sec_minimal_models}. As a straightforward corollary of Theorem \ref{thm_main}, a minimal model for a harmonic odd cyclic \Linf-algebra can be given via a formal super integral. Section \ref{sec_minimal_models} closes by providing an inverse \Linf-algebra morphism to the differential graded Lie algebra morphism embedding the functions on homology into the space of all functions: this is the content of Theorem \ref{thm_inverse}. The combinatorial approach to formal super integration is briefly discussed in Appendix \ref{sec_appendix}. More precisely, within this appendix those formal super integrals considered throughout the paper are shown to admit a presentation as formal sums over stable graphs: Theorems \ref{thm_disconnected_graphs} and \ref{thm_connected_graphs}. The same result is given in \cite{costello}, but proven by different means. Appendix \ref{sec_appendix} closely follows the argument of \cite{etingof}. Indeed, the ordinary notion of a graph is a special case of a stable graph and, by restricting our formula to usual graphs, one can recover that of \cite{etingof}. A very similar result to Theorem \ref{thm_disconnected_graphs} is given in \cite[Example 3.10]{fiorenza_murri}. Further, in loc.~cit.~the authors explore the relationship between Feynman diagram expansions and graphical calculus of Reshetikhin-Turaev.

\section*{Notations and conventions}

Fix the real numbers, $\mathbb{R}$, as the base field. For technical reasons, the base field is extended to the field of formal Laurent series $\mathbb{R}((\hbar))$ at some points of the paper. All unmarked tensors are assumed to be over the appropriate base field, unless otherwise stated. We will be concerned with the category of differential $\mathbb{Z}/2\mathbb{Z}$-graded vector spaces (`super vector spaces'). Some of the definitions and results contained within this paper could also be made sense of in the $\mathbb{Z}$-graded context once suitable adaptations are made. Working with $\mathbb{Z}$-graded objects is, however, of no benefit to this work and would only complicate matters. Of course, any $\mathbb{Z}$-graded object can be regarded as a $\mathbb{Z}/2\mathbb{Z}$-graded object by remembering only the parity of the grading, i.e.~a $\mathbb{Z}$-graded object $V=\bigoplus_{n\in\mathbb{Z}} V_n$ can be regarded as a $\mathbb{Z}/2\mathbb{Z}$-graded object with $V_{\text{even}}=\bigoplus_{n\in \mathbb{Z}} V_{2n}$ and $V_{\text{odd}}=\bigoplus_{n\in \mathbb{Z}} V_{2n+1}$.

The degree (or parity) of a homogeneous element $v$ in a super vector space is denoted $|v|$. Following well established notation those elements of homogeneous degree $0$ are referred to as even and those of homogeneous degree $1$ are referred to as odd. Accordingly, the dimension of a super vector space is given as $(m|n)$, where $m$ is the dimension (in the non-graded sense) of the subspace of even elements and $n$ is the dimension of the subspace of odd elements. The total dimension of a super vector space of dimension $(m|n)$ is given by $m+n$. A super vector space will be finite-dimensional if, and only if, it is of finite total dimension. The tensor product $V\otimes W$ of super vector spaces $V$ and $W$ has differential defined as $d_{V\otimes W}(v \otimes w)=(d_V v)\otimes w + (-1)^{|v|}v\otimes (d_W w)$ and thus the category of super vector spaces is symmetric monoidal with symmetry isomorphism given by $s(v\otimes w) = (-1)^{|v||w|}w\otimes v$.

Denote by $\Pi \mathbb{R}$ the super vector space of total dimension one (over $\mathbb{R}$) concentrated in odd degree. The functor given by taking a super vector space $V$ to the tensor $V\otimes \Pi \mathbb{R}$ is denoted by $\Pi$ (the super vector space $\Pi V$ is called the \emph{parity reversion} of $V$). Likewise for super vector spaces over $\mathbb{R}((\hbar))$. The notation $\Hom(V, W)$ denotes the super vector space with even part the space of morphisms $V \to W$ (i.e.\ those linear maps preserving the grading) and odd part the space of morphisms $V \to \Pi W$ (i.e.\ those linear maps which reverse the grading). This can be equipped with the differential $df=d_W \circ f - (-1)^{|f|}f \circ d_V$, making it into an internal $\Hom$ functor, and hence the category of super vector spaces is a closed symmetric monoidal category.

In particular, an associative, a commutative, or a Lie algebra is always the appropriate notion in the category of super vector spaces. The expressions `differential (super)graded', `commutative differential graded algebra', and `differential graded Lie algebra' are abbreviated to `dg', `cdga', and `dgla', respectively.

Given a dgla $\mathfrak{g}$ and a cdga $A$, recall the tensor product $\mathfrak{g}\otimes A$ possesses a well defined structure of a dgla: the bracket is given on elementary tensors by $[x\otimes a,y\otimes b]=[x,y]\otimes (-1)^{|a||y|} ab$.

The notion of a pseudo-compact super vector space is used extensively within this text. A pseudo-compact super vector space is one given by an inverse limit of super vector spaces of finite total dimension. As such, a pseudo-compact super vector space is equipped with a topology induced by the inverse limit, and thus all linear maps of pseudo-compact super vector spaces are assumed to be continuous. The dual of a pseudo-compact super vector space is, therefore, the topological dual. This has the luxury of always having $(V^*)^*\cong V$ without any finiteness conditions. Similarly, it will always be the case that $(V \otimes V)^*\cong V^* \otimes V^*$ since the tensor product of two pseudo-compact super vector spaces $A = \varprojlim_i A_i$ and $B = \varprojlim_j B_j$ will always be the completed tensor product, in other words $A \otimes B$ is the pseudo-compact super vector space given by $\varprojlim_{i,j} A_i \otimes B_j$. Similarly, if $V$ is a discrete super vector space and $A = \varprojlim_i A_i$ is a pseudo-compact super vector space, the tensor product $A\otimes V$ is always assumed to mean $\varprojlim_i A_i \otimes V$. More details on pseudo-compact objects can be found in the literature \cite{gabriel,keller_yang,vandenbergh}. In particular, it should be noted that the functor $V\mapsto V^*$ gives a symmetric monoidal equivalence between the category of pseudo-compact super vector spaces and the opposite category of super vector spaces. Thus, for example, a pseudo-compact dg algebra is equivalently a dg coalgebra.

The completed symmetric algebra of a finite-dimensional super vector space $V$ is an example of a pseudo-compact cdga (it is the dual of the cofree dg cocommutative coalgebra on $V$) and appears regularly throughout this paper. Explicitly, the completed symmetric algebra is the pseudo-compact algebra $\hat{S}(V):=\prod_{i=0}^\infty S^i(V)$, which is the direct product of symmetric tensor powers (over either $\mathbb{R}$ or $\mathbb{R}((\hbar ))$) of the super vector space $V$. The symmetric algebra $S(V)$ is a subalgebra of $\hat{S}(V)$.

We will often refer to a pronilpotent dgla (or cdga), meaning an inverse limit of nilpotent dglas (or cdgas). Nilpotent here will mean `global' nilpotence: the descending central series stabilises at zero.

\section*{Statement of results}

For convenience and motivation, we give some definitions and state the main theorems of the paper here.

Let $V$ be a super vector space with an odd symmetric, non-degenerate bilinear form. It is important to note that the non-degenerate bilinear form forces $V$ to be finite-dimensional. A quantum \Linf-algebra structure on $V$ is a Maurer-Cartan element in a certain dgla $\mathfrak{h}[\Pi V]\subset\hat{S}\Pi V^*[|\hbar|]$, see Definition \ref{def_h[W]} and Definition \ref{def_qLinf}. Moreover, there exists a filtered quasi-isomorphism of dglas, $\iota$, that induces a bijection of quantum \Linf-algebra structures on $\homology (V)$ to $V$, see Proposition \ref{prop_filtered_quasi_iso}.

\begin{repdefinition}{def_minimal}
Given a quantum \Linf-algebra structure $(V,m)$, the minimal model of $(V,m)$ is a quantum \Linf-algebra $(\homology(V),m^\prime)$ such that $\iota(m^\prime)$ is homotopic to $m$ (as a Maurer-Cartan element in $\mathfrak{h}[\Pi V]$).
\end{repdefinition}

There exists a canonical isotropic subspace $\mathcal{L}_s\subset \Pi V$ (see Section \ref{sec_decomp}) which can be endowed with a non-degenerate quadratic function $\sigma$, see Remark \ref{rem_non-degenerate_form} and the beginning discussion of Section \ref{sec_integrating_MC_elements}.

\begin{reptheorem}{thm_main}
Given a quantum \Linf-structure $(V,m)$, the integral formula
\[
m^\prime = \hbar \log \int_{ \mathcal{L}_s} e^{\frac{m}{\hbar}}e^{\frac{-\sigma}{2\hbar}}
\]
defines a quantum \Linf-algebra $(\homology(V),m^\prime)$ and, moreover, it is the minimal model of $(V,m)$.
\end{reptheorem}

\begin{reptheorem}{thm_inverse}
Let $A$ be a pseudo-compact cdga. The morphism of sets 
\[
\MC (\mathfrak{h}[\Pi V],A)\to\MC (\mathfrak{h}[\Pi \homology (V)],A)
\]
given by mapping $\sum_{i\in I}f_i\otimes a_i $ to the function given by
\[
\hbar\log\int_{\mathcal{L}_s} e^{\frac{1}{\hbar}\sum_{i\in I} f_i\otimes a_i} e^{\frac{-\sigma}{2\hbar}}.
\]
provides the inverse \Linf-morphism to the filtered quasi-isomorphism $\iota$.
\end{reptheorem}

\section{Formal odd symplectic geometry}\label{sec_formal_geo}

A brief account of formal linear odd symplectic geometry is contained within this section. Recall that an odd symplectic super vector space is a super vector space with an odd bilinear form that is also non-degenerate and skew-symmetric. An odd symplectic vector space is considered as a formal odd symplectic manifold. As such, some of the terminology used here reflects the geometric setting and many of the results stated in terms of super vector spaces have known analogues and generalisations. For a general treatment of the BV-formalism see, for example, \cite{khudaverdian_geometry_of_superspace,khudaverdian_BV_and_odd_symp_geo,khudaverdian_nersessian,schwarz}. The notation $V$ will generally be used to denote a super vector space endowed with an odd symmetric bilinear form and $W$ will generally be used to denote an odd symplectic super vector space. As such, a first example of an odd symplectic vector space is the following: let $\omega$ be an odd symmetric form on $V$ that is also non-degenerate, then an odd symplectic form on $\Pi V$ can be given by the formula $\tau (\Pi x, \Pi y)=(-1)^{|x|}\omega (x,y)$ for all $x,y\in V$.

\subsection{Preliminaries}

The algebra of functions on an odd symplectic vector space, $W$, is given by $\hat{S} W^*$. Note that $W$ is a finite dimensional super vector space (due to the non-degeneracy of the bilinear form) and $\hat{S} W^*$ is pseudo-compact. If $W$ has a differential, say $d$, then $\hat{S} W^*$ possesses both a canonical differential obtained from $d$ (which herein is denoted, by an abuse of notation, also $d$) and an operator of order two corresponding to the odd symplectic form (the Laplacian, see Definition \ref{def_BV_laplacian}). Every odd symplectic vector space is of even total dimension, and there exists a canonical basis $\lbrace x_i ,\xi_j \rbrace_{i,j \in\lbrace 1,\dots,n\rbrace}$ for $W^*$ with $x_i$ even and $\xi_i$ odd such that the odd symplectic form is of canonical form: $\omega = \sum_i d_{DR} x_i d_{DR} \xi_i$, where $d_{DR}$ is the de Rham differential.

\begin{definition}
Given an odd symplectic vector space, a Lagrangian subspace is a maximal isotropic subspace, i.e.~a subspace such that the restriction of the odd symplectic form vanishes and is of maximal total dimension with this property.
\end{definition}

\begin{remark}
A Lagrangian subspace must have total dimension half that of the whole space, and any isotropic subspace extends to a Lagrangian one.
\end{remark}

Lagrangian subspaces serve as convenient subspaces for integration and are particularly important for this paper. More details are contained in Section \ref{sec_integration}.

\begin{definition}\label{def_BV_laplacian}
Let $W$ be an odd symplectic vector space and let $\lbrace x_i,\xi_j \rbrace_{i,j\in \lbrace 1,2,\dots, n\rbrace}$ be a basis for $W^*$. The Laplacian acts on formal functions by
\[
\Delta(g)=\sum^{n}_{i=1} \partial_{x_i} \partial_{\xi_i} g.
\]
\end{definition}

\begin{remark}
The definition of the Laplacian does not depend upon the choice of basis in $W$, but the fact that $W$ is finite-dimensional is crucial. Without the assumption that $W$ is finite-dimensional the definition of the Laplacian operator is substantially more involved, affecting the entire paper; without a well-defined Laplacian operator, the quantum master equation (Section \ref{sec_master_equations}) makes no sense, hence the definition of quantum \Linf-algebra (Section \ref{sec_linf}) makes no sense.
\end{remark}

\begin{definition}\label{def_BV_algebra}
A dg BV-algebra is a unital cdga $A$ with an odd differential operator, $\Delta$, of order $2$ such that:
\begin{itemize}
\item $\Delta^2=0=\Delta (1)$; and
\item $d\Delta + \Delta d =0$,
\end{itemize}
where $d$ is the differential of $A$.
\end{definition}

For more details on dg BV-algebras and their generalisation to BV$_\infty$-algebras see \cite{braun_laz_homotopy_BV}, for example. The choice of notation $\Delta$ for the Laplacian and for the odd differential operator in Definitions \ref{def_BV_laplacian} and \ref{def_BV_algebra} is deliberate as the next proposition shows.

\begin{proposition}\label{prop_BV_alg}
The Laplacian defines a dg BV-algebra structure on $\hat{S} W^*$ with differential $d$ and BV-operator $\Delta$. Hence, $\hat{S} W^*$  has the structure of a dg odd Poisson algebra with differential $d+\Delta$ and odd bracket given by
\[
[x,y]=(-1)^{|x|}\Delta (xy)-(-1)^{|x|}\Delta (x)y - x\Delta(y).
\]
\end{proposition}
\begin{proof}
It is a straightforward check to show that $\hat{S} W^*$ is a dg BV-algebra. Further, it is a known fact that given a (dg) BV-algebra, the bracket given in the proposition defines the structure of a (dg) odd Poisson algebra.
\end{proof}

For more details regarding BV-algebras and odd Poisson (or Gerstenhaber) algebras see \cite{roger_gerstenhaber}.

\subsection{Master equations}\label{sec_master_equations}
To make sense of some constructions it is necessary to extend the base field from $\mathbb{R}$ to $\mathbb{R}((\hbar))$ and extend super vector spaces from $W$ to $\overline{W}:=W\otimes_{\mathbb{R}} \mathbb{R}((\hbar ))$. The cdga $\hat{S} \overline{W}^*$ is a dg BV-algebra with differential $d$ and BV operator $\hbar\Delta$, where $d$ and $\Delta$ have been extended $\hbar,\hbar^{-1}$-linearly. Hence, $\hat{S} \overline{W}^*$ is a dg odd Poisson algebra with differential $d+\hbar\Delta$. Note that the symmetric tensors here are taken over $\mathbb{R}((\hbar ))$.

For all functions, $f$, in some pronilpotent ideal, we define the formal power series:
\[
e^f=\sum_{n=0}^\infty \frac{f^n}{n!}.
\]
Likewise, $f$ such that $(f-1)$ is in some pronilpotent ideal, we define the formal power series:
\[
\log (f) = \sum_{n=1}^\infty (-1)^{n+1}\frac{(f-1)^n}{n}.
\]
For any $f$ in some pronilpotent ideal we have $\log (e^f)=f$, and for any $g$ such that $(g-1)$ is in some pronilpotent ideal we have $e^{(\log g)}=g$.

\begin{definition}
Let $m= \sum _{i=0}^\infty \hbar^i m_i \in \hat{S} \overline{W}^*$ be an even function with $m_i\in\hat{S} W^*$ for all $i\geq 0$ such that $\frac{m}{\hbar}$ belongs to a pronilpotent ideal of $\hat{S} \overline{W}^*$. The function $m$ is said to satisfy the quantum master equation (QME) if $(d+\hbar\Delta) e^\frac{m}{\hbar}=0$.
\end{definition}

\begin{proposition}
A solution to the QME is equivalent to a solution to the Maurer-Cartan (MC) equation, i.e.~$\forall m\in \hat{S}\overline{W}^*$ such that $\frac{m}{\hbar}$ lies in a pronilpotent ideal
\[
(d+\hbar\Delta) e^\frac{m}{\hbar}=0 \Leftrightarrow (d+\hbar\Delta) m +\frac{1}{2} [m, m]=0.
\]
\end{proposition}
\begin{proof}
The failure of $\Delta$ to be a derivation is measured by the odd Poisson bracket (as in Proposition \ref{prop_BV_alg}) and so one determines the relationship
\[
(d+\hbar\Delta) e^\frac{m}{\hbar}=\frac{1}{\hbar}e^\frac{m}{\hbar}\left((d+\hbar\Delta) m +\frac{1}{2} [m, m]\right).
\]
The result now follows.
\end{proof}

The phrases `solution to the QME' and `MC element' will be used interchangeably. Writing out the MC equation in terms of the expansion $m_0+\hbar m_1 +\hbar^2 m_2 +\dots$ leads to an equivalent system of equations collecting powers of $\hbar$. The first equation $d(m_0) +\frac{1}{2} [m_0, m_0]= 0$ is the classical master equation (CME) for the function $m_0$ (hence $m_0$ defines a cyclic \Linf-algebra), and the second equation $d(m_1) + \Delta (m_0) + [m_0, m_1]=0 $ defines a \emph{unimodular} \Linf-algebra (see \cite{braun_laz_unimodular,granaker}). In \cite{braun_laz_unimodular} the problem of lifting a solution to the CME to a solution of the QME is addressed and unimodularity plays a key role therein.

\subsection{Strong deformation retracts}\label{sec_decomp}

Let $W$ be an odd symplectic super vector space. Given a strong deformation retract (SDR) of $V=\Pi W$ onto some choice of representatives for the homology of $V$, which is, moreover, compatible with the bilinear form in an appropriate way (see below), one arrives at a canonical choice of isotropic/Lagrangian subspace for $W$. The Lagrangian subspace arrived at in this way is used heavily in Section \ref{sec_integration}. This `cyclic' SDR from a space onto its homology is equivalent to that of a Hodge decomposition, cf.~\cite{chuang_laz_feynman,chuang_laz_hodge}. A Hodge decomposition always exists for a finite-dimensional (super) vector space.

\begin{definition}\label{def_SDR}
Let $V$ and $U$ be two super vector spaces, both equipped with odd symmetric bilinear forms. A cyclic SDR from $V$ to $U$ is a pair of even super vector space morphisms $i\colon U\hookrightarrow V$ and $p\colon V\twoheadrightarrow U$ and an odd linear morphism $s\colon V\to V$ such that:
\begin{itemize}
\item $pi=id_U$;
\item $ds+sd=id_V - ip$;
\item $si=0$; $ps=0$; $s^2=0$;
\item $\langle ix , iy \rangle =\langle x, y\rangle $; $\ker (p) \perp \im (i)$; and $\langle sx,y\rangle=(-1)^{|x|}\langle x , sy\rangle$.
\end{itemize}
\end{definition}

\begin{remark}
Forgetting the last condition, concerning the bilinear forms, in Definition \ref{def_SDR} one obtains the usual, well-established, notion of an SDR of super vector spaces. Since we will always assume an SDR is cyclic in this paper we will suppress the adjective cyclic from now on.
\end{remark}

The conditions $si=0$, $ps=0$, and $s^2=0$ are called the side conditions and are not always included; the reason being that they can be imposed at no cost as the following simple proposition, taken from \cite[Lemma B.6.]{braun_laz_unimodular}, shows.

\begin{proposition}
Let $U,V$ be two super vector spaces equipped with odd symmetric bilinear forms and $(i,p,s)$ be morphisms satisfying all the conditions of an SDR except the side conditions, then $s$ can be replaced with a morphism $s'$ in such a way that the triple $(i,p,s')$ is an SDR.
\end{proposition}
\begin{proof}
If $s$ does not satisfy $si=0$ and $ps=0$, it can be replaced with $\tilde{s}=(ds+sd)s(ds+sd)$. By elementary, yet tedious calculations, the triple $(i,p,\tilde{s})$ now satisfies everything except (possibly) $\tilde{s}^2=0$. Replacing $\tilde{s}$ with $s'=\tilde{s}d\tilde{s}$ means the triple $(i,p,s')$ is an SDR.
\end{proof}

The properties of Definition \ref{def_SDR} ensure that given an SDR of $V$ onto $\homology (V)$, one has a decomposition $V=\im (i)\oplus\im (s)\oplus\im (d)$. Furthermore, one has the orthogonality relations:
\[
\im (i)^{\perp}= \im (s) \oplus \im (d),\quad\im (d)^{\perp}= \im (i) \oplus \im (d),\quad\mathrm{and}\quad\im (s)^{\perp}= \im (i) \oplus \im (s).
\]

Let the bilinear form on $V$ be non-degenerate. The decomposition $V=\homology (V) \oplus \im(s) \oplus \im(d)$ gives rise to a decomposition of the odd symplectic vector space $W=\Pi V$ as $W=\Pi \homology (V) \oplus \Pi \im(s) \oplus \Pi \im(d)$. Therefore, define the subspace $\mathcal{L}_s=\Pi \im(s)$ and notice it is a Lagrangian subspace of $\Pi \im(i)^\perp$.

\begin{proposition}\label{prop_lagrangian_has_even_number_of_odd}
If $\mathcal{L}_s$ is non-zero, then the total dimension of the subspace of odd elements is even.
\end{proposition}
\begin{proof}
One can define an odd symmetric non-degenerate form $(x,y)=\langle x, dy \rangle$ on $\mathcal{L}_s$. Therefore, the form is skew symmetric when restricted to odd coordinates and---by forgetting the grading---the form defines an odd symplectic form on the super vector space of odd coordinates, hence there must be an even number of odd coordinates.
\end{proof}

\section{The BV formalism}\label{sec_bv}

Here the necessary facts concerning the BV-formalism are recalled. For a more in depth exploration of the BV-formalism and BV-geometry see one of the many resources \cite{batalin_vilkovisky_gauge_algebra_quantization,batalin_vilkovisky_quantization_of_gauge_algebras,costello,khudaverdian_semidensities,khudaverdian_geometry_of_superspace,khudaverdian_BV_and_odd_symp_geo,
khudaverdian_nersessian,intro_to_graded_geo_BV_formalism_and_applications,schwarz,severa}.

\subsection{Integration and BV Stokes' Theorem}\label{sec_integration}

The integration over Lagrangian subspaces defined in this section provides an important tool for Section \ref{sec_minimal_models}. This integration is often only formally defined, but in some cases restricts to the standard definition of (super) integration. The reader who is already comfortable with integrals over (finite-dimensional) super vector spaces may wish to skip this section, or at least the early discussion. It should be noted that although we take an axiomatic approach to integration, the integrals considered here could, alternatively, be evaluated using perturbation theory, cf.~Appendix \ref{sec_appendix}.

Let $V$ be a super vector space of finite total dimension $m$ concentrated entirely in even degree, i.e.~a usual vector space. Let $A$ be a non-degenerate, positive definite bilinear form on $V$. We begin with case of the integral
\[
\int_V f e^{-A} \mu,
\]
where $f$ is some polynomial function on $V$. It is well known how to calculate such integrals: taking a change of coordinates, one diagonalises $A$ and calculates
\[
\int_{\mathbb{R}^m} f(x) e^{(-\sum_{i=1}^m x_i^2)} dx_1\dots dx_m
\]
using integration by parts and the Gaussian integral $\int_{\mathbb{R}} e^{-x^2} dx=\sqrt{\pi}$. Next, we wish to extend to super vector spaces. For an odd coordinate $\xi$ on some super vector space one has
\[
\int_\mathbb{R}1 d\xi=0 {\quad\mathrm{and}\quad} \int_\mathbb{R}\xi d\xi =1.
\]
This completely defines odd integration (since odd coordinates square to zero).

Let $V$ now be a super vector space of dimension $(m|n)$ and $A$ is a non-degenerate, positive definite bilinear form. The integral of a polynomial function, $f$, over $V$ is given by
\[
\int_V f e^{-A} \mu:=\int_{\mathbb{R}^n}\int_{\mathbb{R}^m} f e^{-A} dx_1\dots dx_m d\xi_1\dots d\xi_n.
\]
The full machinery of integration over super manifolds is not necessary in this paper. However, for a more general approach to integration over super manifolds see \cite{khudaverdian_nersessian}.

We now wish to make one final extension to the integrals being presented here. Namely, we wish to no longer make the assumption that $A$ is positive definite. For odd coordinates, no longer having this assumption makes no difference and we again concentrate our attention to the case where $V$ is a super vector space of total dimension $m$ concentrated entirely in even degree. Let $A$ now be a non-degenerate bilinear form on $V$. After a change of variables, one has $A= \sum_{i=1}^{k} x_i^2 - \sum_{i=k+1}^m x_i^2$. For $1\leq i \leq k$ the integral
\[
\int_\mathbb{R} f(x)e^{-x_i^2} dx_i,
\]
where $f$ is some polynomial function, exists. It is, therefore, the case when $k+1\leq i \leq m$ when we have an issue, i.e.~the integral
\[
\int_\mathbb{R} f(x)e^{x_i^2} dx_i,
\]
does not exist using standard integration. To rectify this issue, introduce the function
\[
g_f (t) := \int_{\mathbb{R}^m} f(x) e^{(\sum_{i=1}^{k} x_i^2 
+ \sum_{i=k+1}^m (tx_i)^2)} dx_1\dots dx_m,
\]
defined for all non-zero real numbers. Expanding $g_f (t)$ as a formal power series, one can use analytic continuation to define $g_f$ for all non-zero complex numbers. The original integral is, therefore, equal to $g_f (i)$. For example $\int_{\mathbb{R}} e^{x^2} \mu_x=-i\sqrt{\pi}$. It is important to note these integrals only exist formally.

One can integrate $g_{x_i^{2n}} (t)$ by parts to establish a recursive relation, just like in the positive definite case. Using these two recursive relations and normalising the integrals by dividing through by
\[
\int_V e^{-A} \mu
\]
we make the following elementary definition after setting the scene.

Let $W$ be an exact odd symplectic vector space such that $W=L_1\oplus L_2$, where $L_1$ and $L_2$ are Lagrangian subspaces of $W$ and $d\colon L_1\to L_2$ is an isomorphism. A suitable choice of coordinates $\lbrace x_1,\dots, x_k\rbrace$ on the even part of $L_1$ diagonalises the quadratic function $x\mapsto\langle x,dx \rangle=\sum_{i=1}^{j} x_i^2 - \sum_{i=j+1}^k x_i^2$. Let $\epsilon(x_i)=1$ if $x_i$ belongs to the first sum and $\epsilon(x_i)=-1$ if $x_i$ belongs to the second sum. Similarly, there exists coordinates $\lbrace \xi_{k+1},\dots, \xi_n \rbrace$ on the odd part of $L_1$ such that the quadratic function $\xi\mapsto \langle \xi ,d\xi \rangle= -(\xi_{k+1} \xi_{k+2} + \xi_{k+3} \xi_{k+4} + \dots + \xi_{n-1} \xi_n)$, since $L_1$ has an even number of odd coordinates (cf.\ Proposition \ref{prop_lagrangian_has_even_number_of_odd}). For $i\in\lbrace 1, 3, 5,\dots, n-1\rbrace$, the coordinates $\xi_i$ and $\xi_{i+1}$ will be said to be a pairing of odd coordinates.

\begin{definition}\label{def_integral}
For even $x_i\in L_1$,
\[
\int_\mathbb{R} e^{\frac{-1}{2\hbar}\langle x_i,dx_i \rangle} \mu_{x_i}=1,
\]
and, for all $n\in\mathbb{N}$, the recursive relation
\[
\int_{\mathbb{R}} x_i^{2(n+1)} e^{\frac{-1}{2\hbar}\langle x_i,dx_i \rangle} \mu_{x_i} = \epsilon (x_i) (2n+1)\hbar \int_{\mathbb{R}} {x_i}^{2n} e^{\frac{-1}{2\hbar}\langle x_i,dx_i \rangle} \mu_{x_i}
\]
will be called integration by parts (for even coordinates). For a pairing of odd coordinates, $\xi_i$ and $\xi_{i+1}$,
\begin{align*}
\int_{\mathbb{R}^2} e^{\frac{1}{2\hbar}\xi_i \xi_{i+1}} \mu_{\xi_i} \mu_{\xi_{i+1}}:=1,\quad& \int_{\mathbb{R}^2} \xi_i e^{\frac{1}{2\hbar}\xi_i \xi_{i+1} } \mu_{\xi_i} \mu_{\xi_{i+1}}:=0, \\
\int_{\mathbb{R}^2} \xi_{i+1} e^{\frac{1}{2\hbar}\xi_i \xi_{i+1} } \mu_{\xi_i} \mu_{\xi_{i+1}}:=0,\quad& \;\;\mathrm{and}\;\; \int_{\mathbb{R}^2} \xi_i \xi_{i+1} e^{\frac{1}{2\hbar}\xi_i \xi_{i+1} } \mu_{\xi_i} \mu_{\xi_{i+1}}:=\hbar.
\end{align*}
The integral is extended in the obvious manner to polynomial functions $f\in S \overline{W}^*$ over $L_1$ by extending $\hbar,\hbar^{-1}$-linearly and by setting
\[
\int_{L_1} f e^{\frac{-\sigma}{2\hbar}}  :=
\int_{\mathbb{R}}\dots\int_{\mathbb{R}} \left( f e^{\frac{-\sigma}{2\hbar}}\big|_{L_1}\right) \mu_{x_1}\dots\mu_{x_k}\mu_{\xi_{k+1}}\dots\mu_{\xi_n},
\]
where $\sigma$ is the quadratic function corresponding to $?\mapsto\langle ?,d? \rangle$. This integral is a Laurent polynomial in $\mathbb{R}[\hbar, \hbar^{-1}]$, because the integrand is restricted to $L_1$ and the integral is taken over all the coordinates of $L_1$.
\end{definition}

\begin{remark}
In Definition \ref{def_integral} integration by parts is defined for even coordinates. There is no integration by parts defined for odd coordinates, because odd coordinates square to zero.
\end{remark}

\begin{remark}\label{rem_non-degenerate_form}
The above argument for even variables in the case where the bilinear form is not necessarily positive definite is known (in a different guise) as the Wick rotation. The non-degeneracy of $\langle x,dx \rangle$ on $L_1$ means that the integration is against a `Gaussian measure', dealing with convergence issues for even variables, although these integrals need only exist formally: see \cite{hamilton_laz_graph}.
\end{remark}

\begin{proposition}
The integrals of Definition \ref{def_integral} do not depend upon the choice of coordinates.
\end{proposition}
\begin{proof}
The (hidden) normalisation by 
\[
\int_{L_1} e^\frac{-\sigma}{2\hbar}
\]
in Definition \ref{def_integral} causes the integrals to be the ratio of two integrals and, thus, they do not depend upon the choice of coordinates.
\end{proof}

An analogue of Stokes' Theorem is now recalled after a couple of auxiliary results. The more general original result is found in \cite{schwarz}. For an alternative proof using the usual exterior calculus in the linear case see \cite{hamilton_laz_graph}.

\begin{proposition}\label{prop_integral_of_derivative}
Let $\mathcal{L}\subset W$ be a Lagrangian subspace such that $\sigma(y)=\langle y,dy \rangle$ is non-degenerate. For $x_i$ even
\[
\int_{\mathbb{R}} \partial_{x_i} (x_i^m e^{\frac{-1}{2\hbar}\sigma(x_i)}) \mu_{x_i} =0.
\]
For a pairing of odd coordinates, $\xi_i$ and $\xi_{i+1}$, and $f\in SW^*$ any polynomial, for $j\in\lbrace i,i+1 \rbrace$
\[
\int_{\mathbb{R}^2} \partial_{\xi_j} (f e^{\frac{1}{2\hbar}\xi_i \xi_{i+1} }) \mu_{\xi_i} \mu_{\xi_{i+1}} =0.
\]
\end{proposition}
\begin{proof}
In the even case, compute the partial derivative and when $m$ is odd integrate by parts. The odd case is immediate from the calculation.
\end{proof}

\begin{proposition}\label{prop_BV_stokes}
Let $\mathcal{L}\subset W$ be a Lagrangian subspace such that $\sigma(y)=\langle y,dy \rangle$ is non-degenerate and let $f\in S W^*$ be any polynomial, then
\[
\int_{\mathcal{L}} \Delta \left(f e^{\frac{-\sigma}{2\hbar}}\right) =0.
\]
\end{proposition}
\begin{proof}
It follows from computation and Proposition \ref{prop_integral_of_derivative}.
\end{proof}

\subsection{Integrating solutions to the QME}\label{sec_integrating_MC_elements}

Section~\ref{sec_integration} assumes the vector space in question is acyclic, i.e.~$\homology (W)=0$, leading to the existence of a Lagrangian subspace on which $\langle x, dx \rangle$ is non-degenerate, see Remark \ref{rem_non-degenerate_form}. In general, the odd symplectic vector space in question may have non-trivial homology and thus we must now consider this more general case. Recall the decomposition
\[
W=\Pi \im (i) \oplus \Pi \im(i)^\perp
\]
given in Section \ref{sec_decomp} of the finite-dimensional super vector space $W$. Clearly, $\homology (\Pi \im(i)^\perp)=0$, so the preceding results can be applied to $\Pi \im(i)^\perp$. Moreover, recall that one has a canonical choice of Lagrangian subspace $\mathcal{L}_s\subset\Pi \im(i)^\perp$.

\begin{definition}
Given an odd symplectic vector space $W=H(W)\oplus \Pi \im(i)^\perp$, the integral of a polynomial $f\in S\overline{W}^*$ over $\mathcal{L}_s$ is given by $(id_{H(W)}\otimes \int_{\mathcal{L}_s}) (f)\in S(\overline{H(W)}^*)$.
\end{definition}

So far integration has been defined for arbitrary polynomials $f\in S\overline{W}^*$. We will need to integrate certain infinite series in $\hat{S}\overline{W}^*$, however this raises the issue of convergence. For instance, the integral $ \int_\mathbb{R}\left (\sum_{k=1}^\infty \frac{x^{2k}}{h^k} \right ) e^{\frac{-1}{2\hbar}x^2}\mu_x$ does not converge. It will be necessary to be able to integrate exponentials of the form $e^{\frac{f}{\hbar}}$ where $f$ is a formal power series belonging to a certain subspace of $\hat{S}W^*[|\hbar|]$. It turns out that integrals converge when integrating functions of this exponential form, as is now explained.

\begin{definition}\label{def_h[W]}
Introduce the weight grading on the cdga $\hat{S} W^* [|\hbar |]$ by requiring that for a monomial $f\in \hat{S} W^*$ of degree $n$, the element $f\hbar^g$ has weight $2g + n$. Let $\mathfrak{h}[W]$ be the subspace of $\hat{S} W^* [|\hbar |]$ containing those elements of weight grading $> 2$.
\end{definition}

\begin{remark}
Not only is $\mathfrak{h}[W]$ important because it defines a subspace of $\hat{S} W^* [|\hbar |]$ where we can make sense of integration, but it is important because it is where quantum \Linf-algebra structures on $\Pi W$ are defined: see Definition \ref{def_qLinf}.
\end{remark}

\begin{proposition}\label{prop_weight_grading}
For all $i\geq 1$ let $F_i$ be the subspace of $\mathfrak{h}[W]$ given by all vectors of weight grading $\geq i-2$. The filtration $\lbrace F_i \rbrace_{i\geq 1}$ is Hausdorff and complete.\qed
\end{proposition}

It is clear the cdga $\mathfrak{h}[W]$ inherits the structure of an odd Poisson algebra from $\hat{S} W^*[|\hbar |]$. Further, it is clear that $\mathfrak{h}[W]$ is pronilpotent, whereas $\hat{S} W^*[|\hbar |]$ is not.

The following is an elementary observation.

\begin{proposition}\label{prop_fixed_weight_grading}
Integration over an even coordinate or a pairing of odd coordinates fixes the weight grading.
\qed
\end{proposition}

\begin{proposition}\label{prop_integration_is_power_series}
For $f\in\mathfrak{h}[W]$, the formal Laurent series $f^\prime$ given by
\[
f^\prime= \hbar\log \left( \int_{\mathcal{L}_s} e^{\frac{f}{\hbar}} e^{\frac{-\sigma}{2\hbar}} \right)
\]
converges and consists of non-negative powers of $\hbar$ only, i.e.\ it is a formal power series in $\hbar$. Moreover, $f^\prime\in\mathfrak{h}[\homology (W)]$.
\end{proposition}
\begin{proof}
The first statement follows immediately from Appendix \ref{sec_appendix} and, in particular, Theorem \ref{thm_connected_graphs}. The second statement now follows from the first statement and Proposition \ref{prop_fixed_weight_grading}.
\end{proof}

\begin{lemma}\label{lem_qBV}
Let $m\in\mathfrak{h}[W]$ be a solution to the QME, then $m^\prime\in\mathfrak{h}[\homology (W)]$ given by
\[
m^\prime=\hbar \log \left(\int_{\mathcal{L}_s} e^{\frac{m}{\hbar} } e^{\frac{-\sigma}{2\hbar}} \right)
\]
satisfies the QME.\qed
\end{lemma}

The proof is suppressed here as it is a corollary of Proposition \ref{prop_A_linear_BV} given later. Lemma \ref{lem_qBV} is well known in various guises: see \cite{cattaneo_mnev,costello,krotov_losev}. In fact, in loc. cit. the lemma includes an additional statement: if the Lagrangian subspace $\mathcal{L}_s$ is perturbed by a small amount a homotopic solution to the QME is achieved. This latter statement is a corollary of Theorem \ref{thm_main}.

\subsection{Homotopy of solutions to the QME}

In this section, we show that integration respects the notion of homotopy between MC elements. The notion of a homotopy between MC elements in a fixed dgla is standard and is equivalent to gauge equivalence for pronilpotent dglas: see \cite{braun_laz_unimodular,schlessinger_stasheff_deformation_rational,voronov_non-abelian}. Since the dglas considered within this paper are pronilpotent, this section will make the assumption that all dglas are pronilpotent. Given a dgla $\mathfrak{g}$, let $\MC(\mathfrak{g})$ denote the solutions to the MC equation in $\mathfrak{g}$.

\begin{definition}\hfill
\begin{itemize}
\item Let $\mathbb{R}[t,dt]$ be the free cdga over $\mathbb{R}$ with generators $t$ and $dt$ (even and odd respectively) subject to the condition $d(t)=dt$. Clearly there exist two evaluation maps $|_0 , |_1 \colon \mathbb{R}[t,dt]\to \mathbb{R}$ defined by setting $t$ to $0$ and $1$, respectively.
\item For some dgla $\mathfrak{g}$, let $\mathfrak{g}[t,dt]$ be the dgla given by $\mathfrak{g}\otimes \mathbb{R}[t,dt]$. Clearly there exist two evaluation maps $|_0 , |_1 \colon\mathfrak{g}[t,dt]\to \mathfrak{g}$ defined by setting $t$ to $0$ and $1$, respectively.
\end{itemize}
\end{definition}

\begin{definition}
Let $\mathfrak{g}$ be a dgla. $\xi,\eta\in\MC(\mathfrak{g})$ are said to be homotopic if there exists $H(t)\in\MC(\mathfrak{g}[t,dt])$ such that $H(0)=\xi$ and $H(1)=\eta$.

The Maurer-Cartan moduli set, denoted $\mathcal{MC}(\mathfrak{g})$, is the set of equivalence classes of $\MC(\mathfrak{g})$ under the homotopy relation.
\end{definition}

For an odd symplectic vector space $W$, consider a homotopy $H(t)\in\mathfrak{h}[W][t,dt]$. Clearly, $H(t)=A(t)+B(t)dt$. Further, there exists the equivalent exponential form
\[
e^{\frac{H(t)}{\hbar}} = e^{\frac{A(t)}{\hbar}} + \frac{1}{\hbar}e^{\frac{A(t)}{\hbar}}B(t)dt.
\]

We can extend Definition \ref{def_integral} $A$-linearly:

\begin{definition}\label{def_A_linear_integration}
Let $A$ be a cdga. Given an element $\sum_i f_i\otimes a_i\in S W^* \otimes A$, let
\[
\int_{\mathcal{L}_s} \left(\sum_i f_i \otimes a_i \right) e^{\frac{-\sigma}{2\hbar}} := \sum_i \left(\int_{\mathcal{L}_s} f_i e^{\frac{-\sigma}{2\hbar}} \right) \otimes a_i.
\]
\end{definition}

\begin{proposition}\label{prop_laplacian_inside_integral}
If $W=\Pi \im(i)\oplus \Pi \im(i)^\perp$ is an orthogonal decomposition, then the Laplacian splits $\Delta=\Delta_{\Pi \im(i)}+\Delta_{\Pi \im(i)^\perp}$. Moreover, given an element $\sum_i f_i\otimes a_i\in S W^* \otimes A$,
\[
\Delta_{\Pi \im (i)} \int_{\mathcal{L}_s} \left(\sum_i f_i \otimes a_i \right) e^{\frac{-\sigma}{2\hbar}} = \int_{\mathcal{L}_s} \Delta \left(\sum_i f_i \otimes a_i \right) e^{\frac{-\sigma}{2\hbar}}
\]
\end{proposition}
\begin{proof}
The first statement is immediate. The second statement follows from three facts: the observation $\Delta_{\Pi \im(i)}$ commutes with integration and restriction (because it is composed of partial derivatives on $\Pi \im(i)$), the first statement, and Proposition \ref{prop_BV_stokes}.
\end{proof}

\begin{proposition}\label{prop_A_linear_BV}
Let $\sum_i f_i\otimes a_i \in \mathfrak{h}[W] \otimes A$ be a solution to the QME. The function
\[
\hbar \log \left(\int_{\mathcal{L}_s} e^{\frac{1}{\hbar} \sum_i f_i\otimes a_i} e^{\frac{-\sigma}{2\hbar}} \right)
\]
is a solution to the QME in $\mathfrak{h}[\homology (W)]\otimes A$.
\end{proposition}
\begin{proof}
The statement follows readily from Proposition \ref{prop_integration_is_power_series} and Proposition \ref{prop_laplacian_inside_integral}, noticing that the latter continues to hold in the limit.
\end{proof}

As a simple corollary of this proposition it can be seen that integrating a homotopy of MC elements in $\mathfrak{h}[W]$ leads to a homotopy of MC elements in $\mathfrak{h}[\homology (W)]$, by setting $A=\mathbb{R}[t,dt]$.

\begin{proposition}\label{prop_int_homotopy}
Given a homotopy $H(t)$ of $\eta,\nu\in\MC(\mathfrak{h}[W])$, $H^\prime (t)$ given by
\[
H^\prime(t)= \hbar \log \left(\int_{\mathcal{L}_s} e^{\frac{H(t)}{\hbar}} e^{\frac{-\sigma}{2\hbar}}\right)
\]
defines a homotopy of $\eta^\prime,\nu^\prime\in\MC (\mathfrak{h}[\homology (W)])$.\qed
\end{proposition}

In fact, given a specific form of $H(t)$, a specific form of $H^\prime(t)$ can be given.

\begin{proposition}
Given a homotopy $H(t)=A(t)+B(t)dt$ of $\eta,\nu\in\MC(\mathfrak{h}[W])$, the induced homotopy $H^\prime (t)=A^\prime (t)+B^\prime (t)dt$ of $\eta^\prime,\nu^\prime\in\MC (\mathfrak{h}[\homology (W)])$ is given by
\[
A^\prime (t) = \hbar \log \int_{\mathcal{L}_s} e^{\frac{A(t)}{\hbar}} e^{\frac{-\sigma}{2\hbar}}\quad \text{and} \quad B^\prime (t)= e^\frac{-A^\prime (t)}{\hbar} \left(\int_{\mathcal{L}_s} e^\frac{A(t)}{\hbar} B(t) e^{\frac{-\sigma}{2\hbar}}\right).
\]
\end{proposition}
\begin{proof}
A quick check proves the formulae are correct.
\end{proof}

\section{Prerequisites on homotopy Lie algebras}\label{sec_linf}

This section recalls those facts concerning (quantum) \Linf-algebras seen as relevant in the context of this paper, fixing both terminology and notation.

\subsection{Quantum homotopy Lie algebras}

Similar to \Linf-algebra and cyclic \Linf-algebra structures, a quantum \Linf-algebra structure is given by a MC element in a dgla.

\begin{definition}\label{def_qLinf}
Let $V$ be a super vector space equipped with an odd non-degenerate symmetric bilinear form. A quantum \Linf-algebra structure on $V$ is an even element $m(\hbar)=m_0 + \hbar m_1 + \hbar^2 m_2  + \dots \in \mathfrak{h}[\Pi V]$ that satisfies the QME. The pair $(V,m)$ is referred to as a quantum \Linf-algebra.
\end{definition}

It is important to note that a quantum \Linf-algebra structure is defined on a finite-dimensional super vector space endowed with an odd non-degenerate form by a MC element in a pseudo-compact dgla.

\begin{remark}
To elaborate on the weight grading a little, $m_0$ must be at least cubic, $m_1$ must be at least linear, and there are no restrictions on $m_i$ for all $i\geq 2$.

The restriction on the weight grading in Definition \ref{def_h[W]} used to construct $\mathfrak{h}[\Pi V]$ is a reflection of the stability condition for modular operads, cf.~\cite{getzler_kapranov_modular_operads}. Thus one can define quantum \Linf-algebras as algebras over the Feynman transform of the modular closure of the cyclic operad governing commutative algebras. This view is not used here.
\end{remark}

\begin{remark}\label{rem_recovering_defs}
Quantum \Linf-algebras can be thought of as a `higher genus' version of cyclic \Linf-algebras as follows. Given a quantum \Linf-algebra $m(\hbar)= m_0 + \hbar m_1 + \dots$, the canonical derivation associated to $m_0$ defines an odd cyclic \Linf-algebra on $V$ (see \cite{loday}) and, moreover, forgetting the cyclic structure one obtains just a (finite-dimensional $\mathbb{Z}/2\mathbb{Z}$-graded) \Linf-algebra. For more details details regarding \Linf-algebras one should consult the various literature studies: \cite{braun_laz_unimodular,chuang_laz_feynman,hamilton_laz_noncomm_geo,lada_markl}. Further, the derivation associated with $m_0$ paired with the function $m_1$ defines a unimodular \Linf-algebra structure on $V$: see \cite{braun_laz_unimodular,granaker} for details.
\end{remark}

\subsection{Minimal models of quantum homotopy Lie algebras}
Let $(V,m)$ be a quantum \Linf-algebra. Recall there exist two even super vector space morphisms $i\colon \homology (V)\to V$ and $p\colon V\to \homology (V)$ compatible with the odd symmetric forms, coming from an SDR (see Section \ref{sec_decomp}).

\begin{proposition}\label{prop_filtered_quasi_iso}
Given an SDR from $V$ to $\homology (V)$, the dgla morphism 
\[
\iota\colon \mathfrak{h}[\Pi\homology (V)]\to\mathfrak{h}[\Pi V]
\]
defined by $f\mapsto i f p$ is a filtered quasi-isomorphism and hence induces a bijection between the MC moduli sets, i.e.~the homotopy classes of quantum \Linf-algebra structures on $\mathfrak{h}[\Pi\homology (V)]$ and $\mathfrak{h}[\Pi V]$ are in bijective correspondence.
\end{proposition}
\begin{proof}
Taking the filtrations by weight grading, cf.~Proposition \ref{prop_weight_grading}, then clearly one has quasi-isomorphisms
\[
\frac{F_i\mathfrak{h}[\Pi \homology (V)]}{F_{i+1}\mathfrak{h}[\Pi \homology (V)]}\to\frac{F_i \mathfrak{h}[\Pi V]}{F_{i+1}\mathfrak{h}[\Pi V]}.
\]
The fact that filtered quasi-isomorphisms induce isomorphisms of MC moduli sets is well-known (see \cite{braun,getzler}, or the Koszul duality of \cite{laz_markl,maunder_unbased_rat_homo}).
\end{proof}

\begin{remark}
The finer notion of a filtered quasi-isomorphism is required here in order to induce an isomorphism of MC moduli sets. Two dglas which are just quasi-isomorphic may not have isomorphic MC moduli sets as show in the example below.
\end{remark}

\begin{example}
Let $\mathfrak{g}= \left\lbrace a,[a,a]:|a|=-1, da=-\frac{1}{2}[a,a]\right\rbrace$ be a dgla. Clearly $\mathfrak{g}$ is acyclic and quasi-isomorphic to the zero dgla, but $\mathcal{MC}(\mathfrak{g})=\lbrace 0,a\rbrace\neq\mathcal{MC}(0)$.
\end{example}

\begin{definition}\label{def_minimal}
Given a quantum \Linf-algebra $(V,m)$, the minimal model of $(V,m)$ is a quantum \Linf-algebra $(\homology(V),m^\prime)$ such that $\iota(m^\prime)$ is homotopic to $m$ (as a MC element in $\mathfrak{h}[\Pi V]$).
\end{definition}

\begin{remark}
Usually the minimal model for, say, an \Linf-algebra $V$ can be taken to be an \Linf-algebra on the homology $\homology(V)$ which is \emph{\Linf-quasi-isomorphic} to $V$. In the case of quantum \Linf-algebras this does not quite work due to the presence of the non-degenerate bilinear form: there is no longer an especially good notion of a quantum \Linf-map which is not an isomorphism. Therefore, given a quantum \Linf-algebra on the homology $\homology(V)$, we need to say in what sense it is equivalent to the original quantum \Linf-algebra on $V$. In the definition above this is done by extending by zero the quantum \Linf-algebra on $\homology(V)$ to all of $V$ and requiring it to be homotopic (as a Maurer-Cartan element) to the original quantum \Linf-algebra.

It should be noted that if we took this definition in the context of usual \Linf-algebras it would be equivalent to the usual definition, stated at the beginning of this remark.
\end{remark}

Proposition \ref{prop_filtered_quasi_iso} shows the existence and uniqueness up to homotopy of minimal models for quantum \Linf-algebras. However, this does not give an explicit construction or formula for the minimal model. This will be addressed in the next section.

\section{Integral formulae for minimal models}\label{sec_minimal_models}

This section contains the main result of the paper (Theorem \ref{thm_main}): the quantum \Linf-algebra $(\homology (V),m^\prime )$ given by the integral formula in Lemma \ref{lem_qBV} is proven to provide the minimal model for the original quantum \Linf-structure on $(V,m)$.

\subsection{The construction of the minimal model}

Recall from Section \ref{sec_integration} it is possible to integrate a solution to the QME to obtain a solution to the QME on homology.

Viewing integration as a morphism of sets $\rho \colon \MC (\mathfrak{h}[\Pi V])\to\MC (\mathfrak{h}[\Pi \homology (V)])$ it is a one-sided inverse to $\iota\colon \MC (\mathfrak{h}[\Pi \homology (V)])\to\MC (\mathfrak{h}[\Pi V])$, where the restriction of $\iota$ is denoted the same by an abuse of notation.

\begin{proposition}
The morphism $\rho$ is a left inverse of $\iota$. The morphism $\rho$ descends to the level of MC moduli spaces, and therefore induces the inverse bijection on the level of Maurer-Cartan moduli sets.
\end{proposition}
\begin{proof}
It is an easy check to see that $\rho \circ \iota =id_{\MC(\mathfrak{h}[\Pi \homology (V)])}$. To prove the second statement, one recalls Proposition \ref{prop_int_homotopy}.
\end{proof}

\begin{theorem}\label{thm_main}
Given a quantum \Linf-structure $(V,m)$, the integral formula
\[
m^\prime = \rho(m) = \hbar \log \int_{ \mathcal{L}_s} e^{\frac{m}{\hbar}}e^{\frac{-\sigma}{2\hbar}}
\]
defines a quantum \Linf-algebra $(\homology(V),m^\prime)$ and, moreover, it is the minimal model of $(V,m)$.
\end{theorem}
\begin{proof}
The first statement is a rephrasing of Lemma \ref{lem_qBV}. Next, one simply applies the preceding result to see that $\iota(m^\prime)$ is homotopic to $m$.
\end{proof}

\begin{remark}
Using the results of Appendix \ref{sec_appendix}, we can now deduce as a corollary of Theorem \ref{thm_main} the known combinatorial formulae for the minimal model in terms of stable graphs given in \cite{chuang_laz_feynman}.
\end{remark}

Ordinarily, when attempting to lift a solution $m$ to the classical master equation to a solution to the quantum master equation, one is met by a series of obstructions: one would require certain cohomology classes in $\hat{S}\Pi V^*$ to vanish. The first obstruction is just $\Delta(m)$. If this first obstruction is zero not just in cohomology but on the chain level then it turns out, by a straightforward calculation, that all higher obstructions vanish. Such an $m$, with $\Delta(m)=0$, is called harmonic. Summarising, we have the following proposition \cite{braun_laz_unimodular}.

\begin{proposition}
Let $(V,\xi)$ be an odd cyclic \Linf-algebra and let $m$ denote the Hamiltonian function associated to $\xi$. If $m$ is harmonic $(\Delta (m)=0)$, then $(V,m)$ is a quantum \Linf-algebra lifting $\xi$.\qed
\end{proposition}

Therefore, having a quantum lift allows one to apply Theorem \ref{thm_main}, resulting in the following.

\begin{corollary}
Let $(V,\xi)$ be an odd cyclic \Linf-algebra and let $m$ denote the Hamiltonian function associated to $\xi$. If $m$ is harmonic, then the minimal model of $(V,\xi)$ is given by $(\homology(V), X_{\rho (m)})$, where $X_f$ denotes the derivation associated to a Hamiltonian $f$.\qed
\end{corollary}

\subsection{An inverse \Linf-morphism}

The morphism $\iota$ is a (filtered) quasi-isomorphism of dglas, and as such there exists an inverse \Linf-algebra morphism to $\iota$, meaning an \Linf-algebra morphism which induces the inverse to $\iota$ on the level of homology. To construct this morphism a preliminary result is recalled.

\begin{definition}
Let $V$ and $W$ be two dglas. An \Linf-morphism $f\colon V\to W$ is a cdga morphism $\hat{S}\Pi V^* \to \hat{S}\Pi W^*$.
\end{definition}

MC elements of dglas play a significant role in the theory of \Linf-algebras: in addition to being used to define \Linf-algebra structures, they also correspond to morphisms of pseudo-compact cdgas as shown in the following well known result (which can be extended to include \Linf-algebras, cf.\ \cite{braun_laz_homotopy_BV} for example).

\begin{proposition}\label{prop_rep_functor}
Let $V$ be a dgla and $A$ be a unital pseudo-compact cdga. The functor given by taking $A\mapsto\MC (V\otimes A)$ is represented by $\hat{S}\Pi V^*$.\qed
\end{proposition}

\begin{remark}\label{yoneda}
An \Linf-morphism of dglas $V\to W$ gives rise, for any unital pseudo-compact cdga $A$, to a map of sets $\MC(V\otimes A) \to \MC(W\otimes A)$ functorial in $A$, i.e.\ a natural transformation. Moreover, any such natural transformation is, by Yoneda's Lemma, equivalent to having an \Linf-morphism $(V,m_V)\to(W,m_W)$. For greater details see \cite{chuang_laz}.
\end{remark}

Therefore, if the morphism $\rho$ can be extended to include dg coefficients in a functorial manner, this is equivalent to the required \Linf-algebra morphism.

\begin{definition}
Let $A$ be a pseudo-compact cdga. The morphism of sets
\[
\tilde{\rho} \colon \MC (\mathfrak{h}[\Pi V],A)\to\MC (\mathfrak{h}[\Pi \homology (V)],A)
\]
is given by mapping $\sum_{i\in I}f_i\otimes a_i $ to the function given by
\[
\hbar\log\int_{\mathcal{L}_s} e^{\frac{1}{\hbar}\sum_{i\in I} f_i\otimes a_i} e^{\frac{-\sigma}{2\hbar}}.
\]
\end{definition}

The morphism $\tilde{\rho}$ is clearly functorial in both arguments, and by Proposition \ref{prop_A_linear_BV} it is well-defined. Further, $\tilde{\rho}$ is a one-sided inverse to $\iota$ on the level of MC sets. Let $\tilde{\iota}\colon \MC (\mathfrak{h}[\Pi \homology (V)],A)\to \MC (\mathfrak{h}[\Pi V],A)$ be the $A$-linear morphism corresponding to $\iota$, defined in the obvious way.

\begin{theorem}\label{thm_inverse}
$\tilde{\rho}$ is a left inverse of $\tilde{\iota}$. The morphism $\tilde{\rho}$ provides the inverse \Linf-morphism to $\iota$.
\end{theorem}
\begin{proof}
The first statement is straightforward. Using Yoneda's Lemma, as explained in Remark \ref{yoneda}, one arrives at the proof of the second statement.
\end{proof}

\appendix

\section{Integrals as sums over graphs}\label{sec_appendix}

Formal integrals like those used throughout the paper are commonly treated using the formalism of Feynman Diagrams. More precisely, these integrals can often be written as formal series with sums taken over certain graphs. Within this appendix, a presentation as a formal series summing over stable graphs will be given for the integrals considered in this paper. The same presentation is given in \cite[Chapter 2, Section 3]{costello}. Our proof, however, is different and is a mild generalisation of the proof of the case of `usual' graphs given in \cite{etingof}. A very similar result is given in \cite[Example 3.10]{fiorenza_murri}. It should be noted, however, that the obtained formulae in this section are precisely those given in the context of minimal models for algebras over modular operads in \cite{chuang_laz_feynman}. To give the combinatorial presentation, a brief discourse to introduce the relevant material is necessary.

\subsection{Stable graphs}

Stable graphs were introduced by Ginzburg and Kapranov \cite{ginzburg_kapranov} in the context of modular operads and later used in giving formulae for minimal models by Chuang and Lazarev \cite{chuang_lazarev_dual_feynman,chuang_laz_feynman}. Here only the briefest of details will be recalled and for more details one should consult those papers cited.

\begin{definition}
A graph $G$ is given by the following data:
\begin{itemize}
\item A finite set of half edges, $\operatorname{Half}(G)$, and a finite set of vertices, $\operatorname{Vert}(G)$, with a morphism $f\colon\operatorname{Half}(G)\to \operatorname{Vert}(G)$ and an involution $\sigma\colon \operatorname{Half}(G)\to\operatorname{Half}(G)$.  
\item The set of edges, $\operatorname{Edge}(G)$, is the set of two-cycles of $\sigma$ and the legs, $\operatorname{Leg}(G)$, are the fixed points of $\sigma$.
\item For a vertex $v$ the valence is the cardinality of $f^{-1}(v)$, i.e.\ the number of half edges attached to $v$.
\end{itemize}
\end{definition}

\begin{definition}
A stable graph is a graph $G$ such that every vertex $v$ is decorated with a non-negative integer $g(v)$, called the genus of the vertex $v$, such that $2 g(v) + n(v) \geq 3$. The homology $\operatorname{H}_\bullet (-)$ of a stable graph is given by the homology of corresponding one dimensional CW-complex. The genus of a stable graph (denoted $g(G)$) is given by $\dim( \homology_1 (G) ) + \sum_{v\in \operatorname{Vert} (G)} g(v)$.
\end{definition}

\begin{definition}
Given a stable graph $G$ its Euler Characteristic $\chi (G)$ is given by the difference $\dim( \homology_0 (G) ) - g(G)$.
\end{definition}

\begin{example}
For a connected stable graph $G$ where every vertex has genus zero, one recovers the classical result for graphs $\chi (G)= |\operatorname{Vert}(G)|-|\operatorname{Edge}(G)|$. 
\end{example}

\subsection{Feynman Expansions}

The ideas behind the arguments used in this section are largely standard and closely follow those of Etingof \cite{etingof}. Indeed, the formulae of \cite{etingof} can be extracted from our formulae by restricting to those stable graphs where every vertex has genus zero.

Throughout this section $W$ is an odd symplectic vector space with a decomposition $W=\homology (W) \oplus \mathcal{L}_s \oplus \Pi \im(d)$ given by an SDR, see \ref{sec_decomp}. Recall that $W$ is forced to be finite-dimensional by the existence of the non-degenerate bilinear form and that the form $\sigma (?)=\langle ?,d? \rangle$ is non-degenerate on $\mathcal{L}_s$ and denote the inverse form on $\mathcal{L}_s^*$ by $\sigma^{-1}$. Integration over $\mathcal{L}_s$ is given by $id_{\homology (W)} \otimes \int_{\mathcal{L}_s}$.

\begin{definition}
Let $f\in\mathfrak{h}[W]$. Given a  connected stable graph $G$, the Feynman amplitude $F(G)\in (\homology(W)^*)^{\otimes \lvert \operatorname{Leg}(G) \rvert}$ is given as follows:
\begin{itemize}
\item place at every vertex with genus $i$, $j$ half edges, and $k$ legs the component in the symmetric tensor $(\mathcal{L}_s^*)^{\otimes j}\otimes (\homology(W)^*)^{\otimes k}$ of the coefficient of $h^i$ in $f$.
\item take contraction of tensors along each edge using the form $\sigma^{-1}$.
\end{itemize}
If $G$ is disconnected, then $F(G)$ is given by the product of the amplitudes given by its connected components. The empty stable graph has amplitude $1$.
\end{definition}

As is the case in all Feynman expansions, the key is Wick's Theorem.

\begin{theorem}[Wick's Theorem]\label{thm_wick}
Let $\phi_1,\dots ,\phi_{m} \in \mathcal{L}_s^*$. If $m=2k$,
\[
\int_{\mathcal{L}_s} \phi_1\dots \phi_m e^{\frac{-\sigma}{2} }  = \sum_{\mathrm{pairings}}\sigma^{-1} (\phi_{i_1},\phi_{i_2})\dots \sigma^{-1} (\phi_{i_{k-1}},\phi_{i_k}).
\]
If $m$ is odd, the integral is zero.
\end{theorem}
\begin{proof}
For the case $m$ is odd the result is immediate. The case $m$ is even is almost as straightforward, but requires to be broken down further: it is necessary to consider the cases of even integration and odd integration independently. For even integration it suffices to prove the result for $\phi_1=\dots=\phi_m=x$ for a canonical basis element $x$, i.e.\ one can assume $\mathcal{L}_s$ is of total dimension one. The result now readily follows from the definition of integration. For odd integration it, again, reduced to a simple case: it suffices to prove the result for a pairing of canonical odd coordinates $\xi_i$ and $\xi_j$. Since odd elements square to zero it must be the case that $m=2$ and the result is then obvious.
\end{proof}

\begin{theorem}\label{thm_disconnected_graphs}
Let $f\in\mathfrak{h}[W]$, then
\[
\int_{\mathcal{L}_s} e^{\frac{f}{\hbar}} e^{\frac{-\sigma}{2\hbar} } = \sum_{G} \hbar^{-\chi(G)} \frac{F(G)}{|\operatorname{Aut}(G)|},
\]
where the sum is over all (possibly disconnected) stable graphs.
\end{theorem}
\begin{proof}
One can write the restriction of $f$ to $\homology(W)^*\oplus \mathcal{L}_s^*$ as $\sum_{i,j,k} \frac{1}{j!k!} \hbar^i f_{i,j,k} $, where $j$ is the number of linear factors of $\mathcal{L}_s$ in $f_{i,j,k}$ and $k$ is the number of linear factors of $\homology (W)$. Substituting $y \hbar^{\frac{1}{2}}=x$ in $\mathcal{L}_s$ and after expanding $e^{\frac{f}{\hbar}}$ in terms of its Taylor expansion, one can write
\[
\int_{\mathcal{L}_s} e^{\frac{f}{\hbar}} e^{\frac{-\sigma}{2 \hbar} } = \sum_{N} Z_N,
\]
where the sum is over $N=(n_{i,j,k})$, where each $n_{i,j,k}$ is an integer and is zero if $2i+j+k < 3$, and
\[
Z_N = \bigintss_{\mathcal{L}_s} \left( \prod_{i} \prod_{j} \prod_k \frac{\hbar^{\left(i+\frac{j}{2} -1 \right) n_{i,j,k}}}{(j!k!)^{n_{i,j,k}} n_{i,j,k}!} f_{i,j,k}^{n_{i,j,k}} \right) e^{\frac{-\sigma}{2 \hbar} }.
\]
Clearly, every $f_{i,j,k}$ is a product of linear functions and therefore, using Wick's Theorem (\ref{thm_wick}), the integral for each $N$ is given combinatorially as follows: every $f_{i,j,k}$ gives a decorated flower, i.e.\ a vertex with genus $i$, $j$ half-edges, and $k$ legs. One can then choose a pairing, $p$, of half-edges of all flowers and contract using $\sigma^{-1}$ to produce a function $F_p$. Thus
\[
Z_N = \prod_{i} \prod_{j} \prod_k \frac{\hbar^{\left(i+\frac{j}{2} - 1 \right) n_{i,j,k} }}{(j!k!)^{n_{i,j,k}} (n_{i,j,k}!)} \sum_{p} F_p.
\]
A choice of pairing $p$ of half-edges can be visualised as a glueing of the two half-edges in a pair together. Thus, a gluing creates a stable graph $G$ and $F_p$ is, in fact, precisely the Feynman amplitude $F(G)$. What's more, it is clear that any stable graph with $n_{i,j,k}$ vertices of genus $i$, valence $j$, and having $k$ legs can be obtained from a pairing in this way. Since the goal is to sum over stable graphs, one must take care to deal with the redundancies arising from the fact that the same graph can be obtained in multiple ways from different pairings of half-edges. To this end, consider the permutations of half-edges that preserve decorated flowers. This group of permutations involves three parts: the permutations of flowers with a given genus, valence, and number of legs; permutations of the half-edges of a flower; and permutations of the legs of a flower. To be precise, the group is the semi-direct product
\[
\prod_i \prod_ j \prod_k S_{n_{i,j,k}} \ltimes (S_j^{n_{i,j,k}}\times S_k^{n_i,j,k}),
\]
which has order $\prod_i \prod_j \prod_k (j!k!)^{n_{i,j,k}} (n_{i,j,k}!)$. It is clear to see that the group acts transitively on all pairings of half-edges that result in a particular stable graph and the stabiliser of such a pairing is the automorphism group of the resulting graph. Therefore, the number of pairings resulting in a stable graph $G$ is given by $\frac{\prod_i \prod_j \prod_k (j!k!)^{n_{i,j,k}} (n_{i,j,k}!)}{|\operatorname{Aut}(G)|}$. Putting this all together, the result is obtained.
\end{proof}

Notice the series on the right hand side of Theorem \ref{thm_disconnected_graphs} involves arbitrary (possibly negative) powers of $\hbar$. This is to be expected since $e^{\frac{f}{\hbar}}$ has arbitrary powers of $\hbar$ and integration fixes the weight grading. These negative powers come from disjoint unions of graphs such as the one with two vertices, one edge, and four legs (two at each vertex). Taking the logarithm of the series in Theorem \ref{thm_disconnected_graphs} has the initially surprising effect of reducing the sum to over connected stable graphs, and thus multiplying by $\hbar$ results in only non-negative powers of $\hbar$.

\begin{theorem}\label{thm_connected_graphs}
\[
\hbar \log \int_{\mathcal{L}_s} e^{\frac{f}{\hbar}} e^{\frac{-\sigma}{2 \hbar} } = \sum_{G\mathrm{~connected}} \hbar^{g(G)} \frac{F(G)}{|\operatorname{Aut}(G)|},
\]
where the sum is over all connected stable graphs.
\end{theorem}
\begin{proof}
Writing any disconnected graph as $G=G_1^{k_1} \dots G_l^{k_l}$, for non-isomorphic connected graphs $G_i$, it is clear that $F (G) =F( G_1 )^{k_1} \dots F (G_l )^{k_l}$ and $\chi (G) = k_1 \chi(G_1) + \dots + k_l \chi (G_l)$. Further, $|\operatorname{Aut}(G)|= \prod_i \left( |\operatorname{Aut}(G_i))|^{k_i} ({k_i}!) \right)$. Thus, exponentiating the series
\[
\frac{1}{\hbar} \sum_{G\mathrm{~connected}} \hbar^{g(G)} \frac{F(G)}{|\operatorname{Aut}(G)|},
\]
one arrives at the series of Theorem \ref{thm_disconnected_graphs}.
\end{proof}

\section*{Acknowledgements}

The authors would like to thank Andrey Lazarev for many useful suggestions, conversations, and for posing the project. Thanks is also owed to the useful comments provided by Paul Levy, Jim Stasheff, Ted Voronov, and the anonymous referee.

\bibliographystyle{plain}
\bibliography{my_bib}

\end{document}